\documentclass[12pt,a4paper]{article}
\usepackage{amsthm}
\usepackage{epsfig}
\usepackage{alltt}
\usepackage[latin1]{inputenc}
\usepackage{makeidx}
\usepackage{newlfont}
\usepackage{amsmath}
\usepackage{amssymb}
\usepackage{amsfonts}
\usepackage{amscd}
\usepackage{enumerate}
\usepackage{mathrsfs}
\usepackage{ifthen}

\newtheorem{theorem} {\bf Theorem} [section]

\newtheorem{prop}[theorem]{\bf Proposition}
\newtheorem{corollary}[theorem]{\bf Corollary}

\theoremstyle{remark}

\def\phi{\varphi}
\def\epsilon{\varepsilon}
\def\theta{\vartheta}

\newcommand{\mcm}{{\mathscr{M}}}

\newcommand{\mcn}{{\mathscr{N}}}

\newcommand{\graph}{{\mathrm{graph}}}

\newcommand{\lspan}{\operatorname{\mathrm span}}

\newcommand{\loc}{\mathrm{loc}}

\newcommand{\Bl}[1]{{\mathbb{#1}}}

\newcommand{\DR}{\Bl{R}}
\newcommand{\DZ}{\Bl{Z}}

\newcommand{\DN}{\Bl{N}}

\newboolean{comment}
\newcommand{\comline}[1]{\ifthenelse{\boolean{comment}}{{\bf
      \noindent\shortstack[r]{\rule{0cm}{0.5cm}\\
      #1\\\rule{16cm}{0.02\pagestyle{myheadings}
 \markboth{Links}{Rechts}
 5cm}}\\}}}

\newcommand{\comtxt}[1]{\ifthenelse{\boolean{comment}}{{\bf #1} \\}}

\newcommand{\commar}[1]{\ifthenelse{\boolean{comment}}{\marginpar{{#1}}}{}}

\setboolean{comment}{true}          

\newcommand{\btxvolumelong }[1]{\relax}

\newcommand{\btxofserieslong }[1]{\relax}

\newcommand{\btxandlong }[1]{and}

\begin{document}

\author{Hannes Junginger-Gestrich \&
Enrico Valdinoci\thanks{Addresses:
HJG,
{\sf
Mathematisches Institut,
Albert-Ludwigs-Universit\"at Freiburg,
Abteilung f\"ur Reine Mathematik,
Eckerstra{\ss}e 1,
79104 Freiburg im Breisgau (Germany)},
{\tt
hannes.junginger-gestrich@math.uni-freiburg.de},
EV,
{\sf Di\-par\-ti\-men\-to di Ma\-te\-ma\-ti\-ca,
Uni\-ver\-si\-t\`a di Ro\-ma Tor Ver\-ga\-ta,
Via della Ri\-cer\-ca
Scien\-ti\-fi\-ca~1, 00133~Roma (Italy)},
{\tt enrico@mat.uniroma3.it}.
The work of EV was
supported by MIUR Variational
Methods and
Nonlinear Differential Equations.
Diese Zusammenarbeit wurde bei einem sehr angenehmen Besuch von EV
in Freiburg
begonnen.}}

\title{Some connections
between results and
problems of
De Giorgi, Moser and Bangert}

\maketitle

\begin{abstract}
  Using theorems of Bangert, we prove a rigidity result which shows how a
  question raised by Bangert for elliptic integrands of
Moser type is connected, in 
the case of minimal solutions
  without self-intersections, to a famous conjecture of De Giorgi for phase
  transitions.
\end{abstract}

\section{Introduction}

The purpose of this note is to relate some probelms posed by Moser \cite{Mos},
Bangert \cite{Ba89} and De Giorgi \cite{DG}.  In particular, we point out that a
rigidity result in a question raised by Bangert for the case of minimal
solutions of elliptic
integrands would imply a one-dimensional symmetry for minimal phase transitions
connected to a famous conjecture of De Giorgi.

Though the proofs we present here are mainly a straightening of the existing
literature, we hope that our approach may clarify some points in these important
problems and provide useful connections.

\subsection{The De Giorgi setting}\label{DG:set}
 
A classical phase transition model (known in the literature
under the names of Allen, Cahn, Ginzburg, Landau, van der Vaals, etc.)
consists in the study of the elliptic equation
\begin{equation}\label{eq1}
\Delta u = u-3u^2 +2u^3\,,
\end{equation}
where $u\in C^2(\DR^n)$.  Particular solutions of \eqref{eq1} are the local
minimizers of the associated energy functional. Namely, we define $W(u):=u^2
(1-u)^2$ and we say that $u\in C^2(\DR^n,(0,1))$ is a minimal solution of
\eqref{eq1} if
\begin{equation}\label{eq2}
\int_B |u_x|^2+W(u)\,dx\leq\int_B |u_x+\phi_x|^2+W(u+\phi)\,dx
\end{equation}
for any $\phi\in C^\infty_0(B)$ and any ball $B\subset \DR^n$.

Following is a celebrated question\footnote{To make the
notation of this note uniform, we allow ourselves to slighlty change
the notation of \cite{DG}: namely, what here is $u$, there is
$2u-1$, so that the range of $u$, which is here $(0,1)$, corresponds
to $(-1,1)$ there.} 
posed in 
\cite{DG}:

{\bf Problem \cite{DG}:}
{\em Let $u \in C^2(\DR^n)$ be a solution of
\eqref{eq1} in the whole $\DR^n$. Suppose that $0<u(x)<1$ and
$\partial_n u(x)>0$ for any $x\in\DR^n$.

Is it true that all the level sets of $u$ are hyperplanes, at least if $n\leq
8$?}

To the best of our knowledge, this problem is still open in its generality, and
a complete answer is known only if $n=2$ \cite{BCN, GG} and $n=3$ \cite{AC}.  In
these cases, indeed, the answer to the above question is positive in a much more
general setting \cite{AAC} -- in particular, no structural assumptions are
needed for the nonlinearity on the right hand side of \eqref{eq1}.  When $4\leq
n\leq 8$, the conjecture has been proven \cite{S} under the additional
assumptions that
$$ \lim_{x_n\rightarrow-\infty}u(x',x_n) =0
\;{\mbox{ and }}\;
\lim_{x_n\rightarrow+\infty}u(x',x_n) =1\,.$$
If the above limits are uniform, the conjecture holds
in any dimension $n$ \cite{Fa, BHM, BBG}.

The problem has also been dealt with
for $p$-Laplacian-type operators
\cite{Fap, DaG, VSS}, in the {H}eisenberg group framework
\cite{BL} and for free boundary models
\cite{VMZ}.

A natural question arising from \cite{DG} is whether analogous statements hold
for minimal solutions. We state this question in the following form: \medskip

{\bf Problem \cite{DG}$_{MIN}$:} {\em Let $u \in C^2(\DR^n,(0,1))$ be a
minimal solution of \eqref{eq1}.
Is it true that all the level
  sets of $u$ are hyperplanes, at least if $n$ is small enough?}
\medskip

The answer to the above question is known to be positive for $n\le 7$ \cite{S}.
We will see that
Problem {\bf \cite{DG}$_{MIN}$} has some relation with another
one, posed by \cite{Ba89} for minimizers without self-intersections in the
periodic elliptic integrand context.

\subsection{The Moser-Bangert setting}\label{Ba:set}

Given $F:\DR^n\times\DR\times\DR^n\to \DR$, which is
$\DZ$-periodic in the first $n+1$ variables,
one studies functions $u:\DR^n \to
\DR$ that minimize the integral
$\int F(x,u,u_x)\,dx$ with respect to compactly supported
variations, that is
\begin{equation}
  \label{eq:mos_min}
  \int_B F(x,u,u_x)\,dx
  \leq
  \int_B F(x,u+\phi,u_x+\phi_x)\,dx
  \,,  
\end{equation}
for any $\phi\in C^\infty_0(B)$ and for any ball $B\subset\DR^n$.

We assume $F \in C^{2,\epsilon}(\DR^{2n+1})$, with some 
$\epsilon\in(0,1]$. We also suppose that
$F=F(x,u,p)$ satisfies the following
appropriate growth conditions (compare with \cite[(3.1)]{Mos}):
\begin{equation}\label{Ba88et}
\begin{split}
&\frac 1 c\,
|\xi|^2\leq \sum_{1\leq i,j\leq n}F_{p_i p_j}(x,u,p) \xi_i \xi_j
\leq c\,|\xi|^2\,\\
&|F_{pu}|+|F_{px}|\leq c(1+|p|)\,,\\
&|F_{uu}|+|F_{ux}|+
|F_{xx}|\leq c(1+|p|^2)\,,
\end{split}
\end{equation} 
for any $(x,u,p)\in\DR^n\times\DR\times\DR^n$ and any $\xi\in\DR^n$,
for a suitable $c\geq 1$.

The above assumptions ensure the ellipticity of the corresponding Euler-Lagrange
equation.  Under these conditions, the minimizers inherit regularity from $F$
and they are of class $C^{2,\epsilon}(\DR^n)$ (see \cite[page 246]{Mos} for
further details).
  
If $u:\DR^{n+1} \to \DR$ and $\bar k = (k,k_{n+1}) \in \DZ^{n+1}$, we
define\footnote{We will often adopt the notation of writing barred vectors for
  elements of $(n+1)$-dimensional spaces: e.g., $k\in\DZ^n$ versus $\bar
  k\in\DZ^{n+1}$.}  $T_{\bar k}u:\DR^n\to \DR$ as
$$T_{\bar k}u(x) = u(x-k)+ k_{n+1}\,.$$
Since $F$ is $\DZ^{n+1}$-periodic, $T$ determines a $\DZ^{n+1}$-action on the
set of minimizers.

We will consider the partial ordering on the set of functions for which we say
that $u < v $ if and only if $u(x) < v(x)$ for all $x \in \DR^n$.  We then look
at minimizers {\em without self-intersections}, i.e. minimizers whose $T$-orbit
is totally ordered with respect to the above partial ordering. More explicitly,
we say that a minimizer $u$ is without self-intersections if $T_{\bar k} u$ is either
$>$, $<$ or $=$ $u$. It is readily seen that a minimizer $u$ is without
self-intersections if and only if the hypersurfaces $\graph (u) \subset
\DR^{n+1}$ have no self-intersections when projected into the standard torus
$\DR^{n+1}/\DZ^{n+1}$ (and this property justifies the name given to it).

One denotes the set of minimizers without self-intersections by $\mcm$.  For
every $u \in \mcm$, \cite[Theorem 2.1]{Mos} shows that $\graph(u)$ lies within
universally bounded Hausdorff distance from a hyperplane: more explicitly, there
exists $C\ge 0$ such that for every $u \in \mcm$ there exists $\rho\in\DR^n$
with
\begin{equation}\label{AAJ}
|u(x)-u(0)-\rho\cdot x|\leq C
\end{equation}
for any $x\in\DR^n$.

We set $\bar a_1(u)=\bar a_1:=(-\rho,1)/\sqrt{|\rho|^2+1}\in\DR^{n+1}$.
Geometrically, $\bar a_1(u)$ is the unit normal to the above mentioned
hyperplane which has positive inner product with the $(n+1)$st standard
coordinate vector.  We recall that $\bar a_1(u)$ is sometimes called {\em
  rotation vector} or {\em average slope} of $u$ (the names are borrowed by
analogous features in dynamical systems, see, e.g., \cite{Mather}).

We now briefly recall some useful {\em invariants} introduced by \cite{Ba89}.
To this extent we remark that if $\bar k\in\DZ^{n+1}$ and $\bar k \cdot \bar
a_1$ is $>0$ ($<0$, respectively), then $T_{\bar k}u> u$ ($<u$, respectively).
To see this, take $\bar k\in\DZ^{n+1}$ with $\bar k \cdot \bar a_1>0$ and
suppose, by contradiction, that $T_{\bar k}u {\not>} u$. Then, since $u$ is
non-self-intersecting, $T_{\bar k}u \le u$ and so $u(x-\ell k)+\ell k_{n+1}\le
u(x)$, for any $\ell\in\DN$.  We thus have
$$0\leq u(\ell k)-u(0)
-\ell k_{n+1}\leq C-\ell (k_{n+1}-k\cdot \rho)\,,$$ thanks to \eqref{AAJ}.  By
taking $\ell$ large, since $k_{n+1}>k\cdot\rho$, one reaches the contradiction
that proves the above observation.

If, on the other hand, $\bar k \cdot \bar a_1 = 0$, it is possible that $T_{\bar
  k}u > u$ or $<u$ or $=u$. Bangert gives a complete description of such
possibilities in \cite[(3.3)-(3.7)]{Ba89}. We subsume this classification as
follows:

\begin{prop}\label{graphbasic}
  For every $u \in \mcm$ there exists an integer $t = t(u) \in \{1,\ldots,n+1\}$
  and unit vectors $\bar a_1= \bar a_1(u),\ldots,\bar a_t = \bar a_t(u)$ such
  that for $1\le s \le t$ we have
  \begin{align}\label{adm}
    \begin{split}
      \bar a_s \in \lspan \bar \Gamma_s\,, \quad \mbox{where } \bar \Gamma_s =
      \bar \Gamma_s(u) := \DZ^{n+1} \cap \big(\lspan\{\bar a_1,\ldots, \bar
      a_{s-1}\}\big)^\bot\,,
    \end{split}
  \end{align} and the $\bar a_1,\ldots,\bar a_t$ are uniquely determined by the
  following properties:
  \begin{enumerate}[{\em (i)}]
  \item $T_{\bar k} u > u$ if and only if there exists $1 \le s \le t$ such that
    $\bar k \in \bar \Gamma_s$ and $\bar k \cdot \bar a_s > 0$.
  \item $T_{\bar k}u = u$ if and only if $\bar k \in \bar \Gamma_{t+1}$.
  \end{enumerate}
\end{prop}

Since, as proved in \cite[Theorem 5.6]{Mos}, if $|\bar a_1| =1$ and $\bar a_1
\cdot \bar e_{n+1} > 0$, there always exist functions $u \in \mcm$ with $a_1(u)
= \bar a_1$, we have that the set to which the above result applies is
non-empty.

A system of unit vectors $(\bar a_1,\ldots, \bar a_t)$ is called {\em
  admissible} if $\bar a_1 \cdot \bar e_{n+1} >0$ and relation \eqref{adm} is
satisfied. For an admissible system $(\bar a_1,\ldots,\bar a_t)$ one writes
$$\mcm(\bar a_1,\ldots,\bar a_t) = 
\big\{u \in \mcm \mid t(u) = t\, \mbox{and } \bar a_s(u) = \bar a_s\;\mbox{for }
1\le s \le t\big \}\,.$$ Many results in the above setting have been obtained by
\cite{Ba89} and some of them will be needed in the sequel.  For instance, in the
following Proposition \ref{4.2}, we recall that for a given solution $u$, there
exist ``envelopping'' solutions $u^-$ and $u^+$ of higher periodicity:

\bigskip\bigskip

\label{fig1}
\vbox{
\centerline{\epsfxsize=4truein \epsfbox{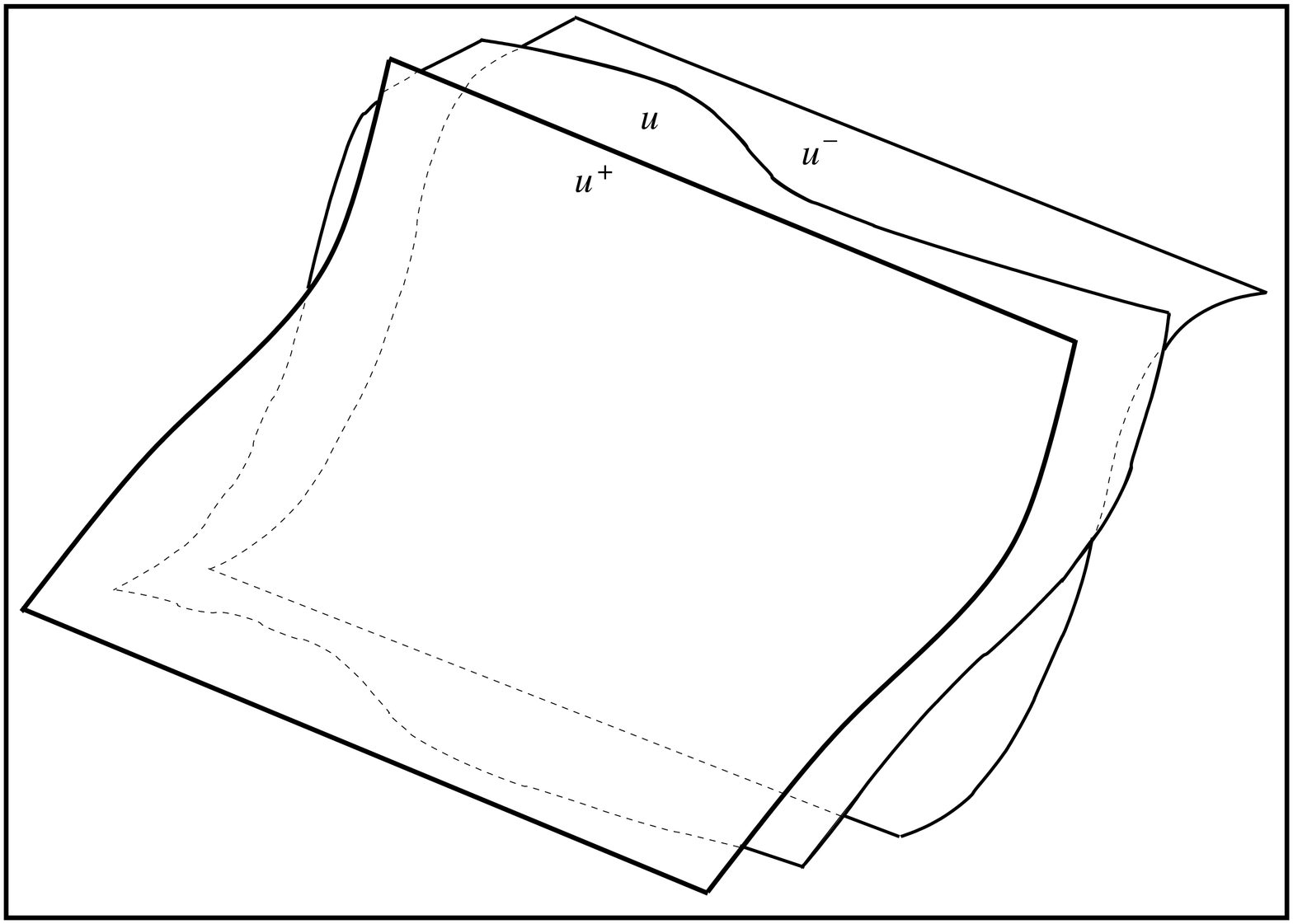}}
\bigskip

\nopagebreak

\centerline{\footnotesize\bf 
Envelopping solutions $u^\pm$.}\bigskip\bigskip}

\begin{prop}\label{4.2}
  If $u \in \mcm(\bar a_1,\ldots, \bar a_t)$ and $t > 1$, then there exist
  functions $u^-$ and $u^+$ in $\mcm(\bar a_1, \ldots \bar a_{t-1})$ with the
  following properties:
  \begin{enumerate}[{\em (a)}]
  \item If $\bar k_i \in \bar \Gamma_t$ and $\displaystyle\lim_{i \to \infty}
    \bar k_i \cdot \bar a_t = \pm \infty$ then $\displaystyle\lim_{i \to \infty}
    T_{\bar k_i} = u^\pm$\,,
  \item $u^-<u<u^+$ and $T_{\bar k} u^- \ge u^+$ if $k \in \bar \Gamma_s$ and
    $\bar k \cdot \bar a_s > 0$ for some $1\le s <t$.
  \end{enumerate}
\end{prop}

For the proof, see \cite[Proposition (4.2)]{Ba89}. A completely satisfactory
uniqueness result in this framework is
\begin{theorem}[\cite{Ba89}, Theorem (6.22)]\label{order}
  If $(\bar a_1, \ldots, \bar a_t)$ is admissible, then the disjoint union
  $\mcm(\bar a_1,) \cup \mcm(\bar a_1, \bar a_2) \cup \ldots \cup \mcm(\bar a_1,
  \ldots, \bar a_t)$ is totally ordered.
\end{theorem}

We point out that the proof of this theorem heaviliy rests on the following
result, which may be seen as a uniqueness (or gap-like) statement for $u^-$ and
$u^+$ in $\mcm(\bar a_1, \ldots \bar a_{t-1})$. The proof of this result is
incomplete in \cite{Ba89} and a completion is given in \cite{Ju2}.
\begin{theorem}[\cite{Ba89}, Thoerem (6.6)]\label{Morse}
  If $u \in \mcm(\bar a_1, \ldots, \bar a_t)$ and $t > 1$, then there does not
  exist $v \in \mcm(\bar a_1, \ldots, \bar a_{t-1})$ such that $u^- < v < u^+$.
\end{theorem}

Bangert posed a deep question in this framework in the very last paragraph of
\cite{Ba89}: \medskip

{\bf Problem \cite{Ba89}:} {\em Is it true that if $u$ is a minimal solution and
  there exist $C\ge0$ and $\rho\in\DR^n$ with $|u(x)-u(0)-\rho\cdot x|\leq C$
  for any $x\in\DR^n$, then $u$ must be without self-intersections?  }\medskip

We recall that the above question is known to have a positive answer when $\bar
a_1$ is rationally independent, cf.~\cite[Theorem (8.4)]{Ba89}.  Remarkably, the
connection between Problems {\bf \cite{DG}$_{MIN}$} and {\bf \cite{Ba89}} will
happen exactly when $\bar a_1=\bar e_{n+1}=(0,\dots,0,1)$, which is rationally
dependent.

The last notion we need to recall is the one of {\em foliation}.  We say that a
connected open subset $G\subseteq\DR^{n+1}$ is foliated by a subset $\mcn
\subseteq C^0(\DR^n)$ if
\begin{align*}
  \graph (u) \cap \graph(v) &= \emptyset \quad \mbox{for all } u,v \in \mcn \quad\mbox{and}\\
  \bigcup_{u \in \DN} \graph (u) &= G\,.
\end{align*}

\section{Our result}

We are now in position to state the rigidity result which is the main purpose of
this note.  On the one hand, as we will see, the proof of it will be a simple
application of the deep results already available in the existing literature.
On the other hand, this result will bridge the problem of De Giorgi with the one
of Bangert.

\begin{theorem}\label{ba_dg} Let the setting of \S\ref{Ba:set} hold
  and let $t\in\DN$, $t\geq 2$.  Suppose that $(\bar a_1, \ldots, \bar a_t)$ is
  admissible and $u_1$, $u_2 \in \mcm(\bar a_1, \ldots, \bar a_{t-1})$ and
  $u_1<u_2$.  Assume that the set
\begin{equation}\label{Gde}
  G = \Big\{(x,x_{n+1})\mid u_1(x) < x_{n+1} < u_2(x)\Big\}
\end{equation}
is foliated by a one-parameter family of functions $(v_b)_{b\in \DR} \subset
\mcm(\bar a_1, \ldots, \bar a_t)$.  Then
\begin{enumerate}[\em (I)]
\item\label{I} $(v_b)^- = u_1$ and $(v_b)^+= u_2$ for every $b\in \DR$.
\item\label{II} If $u \in \mcm$ with $u_1 < u < u_2$, then $t(u) \ge t$ and $\bar a_i(u) =
  \bar a_i$ for any $1\le i < t$. If furthermore $\bar a_t (u) = \bar a_t$, then
  there exists $b_0 \in \DR$ such that $u = v_{b_0}$, and in particular $t(u) =
  t$.
\end{enumerate}
\end{theorem}

In a verbose mode, Theorem \ref{ba_dg} says the following: take an admissible
set of invariants $(\bar a_1, \ldots, \bar a_t)$ and two minimal solutions
$u_1<u_2$ without self-intersections with the above invariants except for the
last one.  Suppose that the space in between $u_1$ and $u_2$ is foliated by
minimizers $v_b$ which have all invariants $(\bar a_1, \ldots, \bar a_t)$.
Then, any minimal solution
without self-intersections lying between $u_1$ and $u_2$ and possesing
the right
last invariant must agree with one of the $v_b$'s.  \bigskip\bigskip

\label{fig2}
\vbox{
\centerline{\epsfxsize=4truein \epsfbox{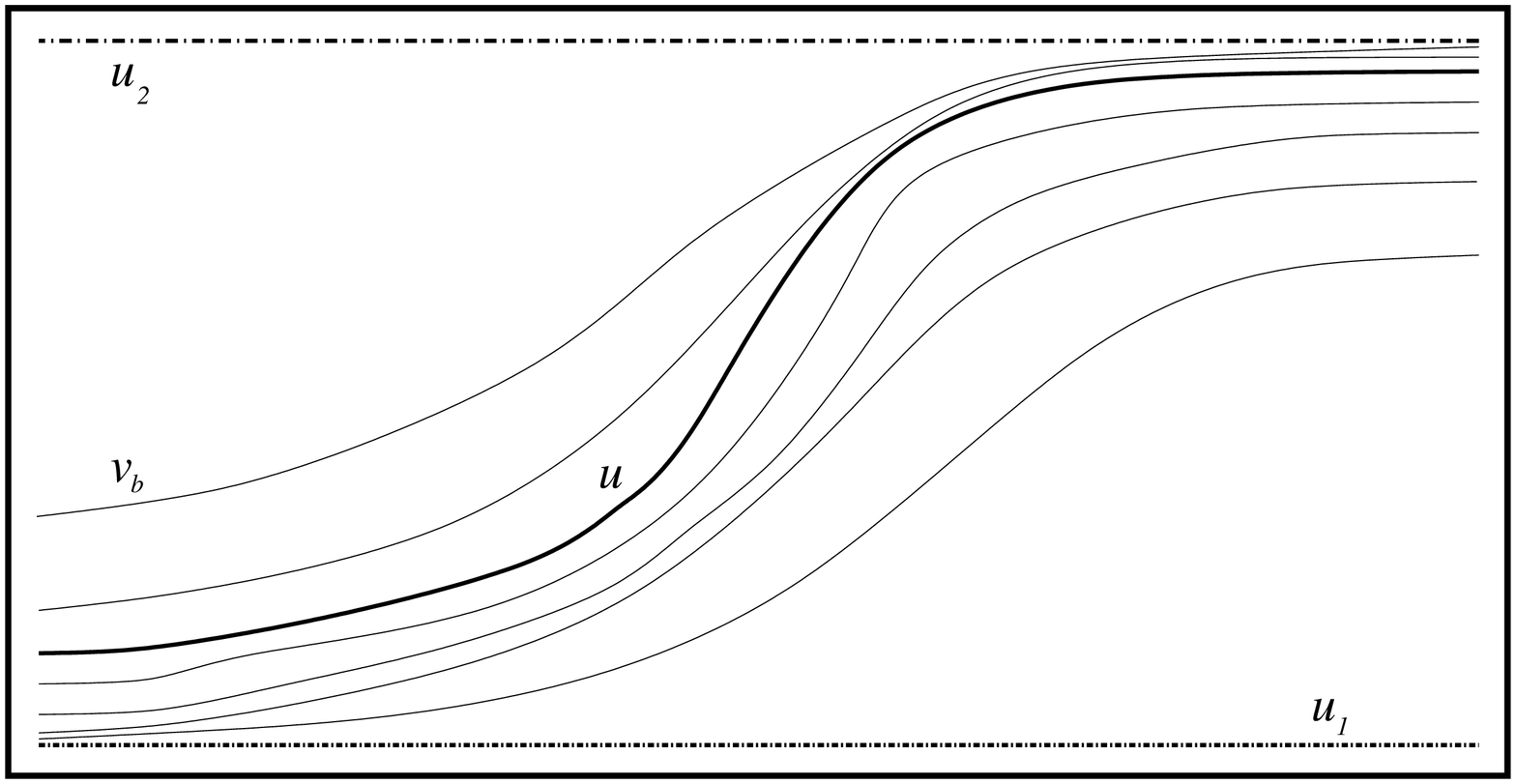}}
\bigskip

\nopagebreak

\centerline{\footnotesize\bf 
The foliation of Theorem \ref{ba_dg}.
}\bigskip\bigskip}

The following results relate Problems {\bf \cite{DG}$_{MIN}$}
and {\bf \cite{Ba89}}
in the case of minimal solutions without
self-intersections of phase transition models:

\begin{corollary}\label{ba_dg-1}
  Let $u\in C^2(\DR^n,(0,1))$ be a minimal solution of \eqref{eq1}.  Suppose
  that $T_{\bar k} u$ is either $>$, $<$ or $=u$ for any $\bar k \in\DZ^{n+1}$.
  Then all the level sets of $u$ are hyperplanes.
\end{corollary}

We denote by $\lfloor r\rfloor$ the integer part of $r\in\DR$.
We then extend the potential of \eqref{eq2}
into a (reasonably smooth) periodic one, in order to connect
the setting in \S\ref{DG:set}
with the one in \S\ref{Ba:set}.

\begin{corollary}\label{ba_dg-2}
  If Problem {\em \bf \cite{Ba89}} has a positive answer in dimension $n$ for
  $F=|p|^2+W(\lfloor u\rfloor)$ and $\bar a_1=\bar e_{n+1}$, then Problem
  {\em \bf \cite{DG}$_{MIN}$} has a positive answer in dimension $n$.
\end{corollary}

We point out that while the setting in
\S\ref{Ba:set} only has discrete translational invariance, the one in
\S\ref{DG:set} possesses full translational and rotational invariance:
thus, in concrete 
cases, other phase transition models
may be reduced to the setting in \S\ref{Ba:set} after appropriate
rotations and scalings.

In the De Giorgi framework, the analysis of
the profile at infinity has often played a central r\^ole
(see, e.g., \cite{AC, S, VSS}). Theorem \ref{ba_dg}
also gives some information about such asymptotic profile,
according to the following result:

\begin{corollary}\label{asy}
  Suppose that $u_1, u_2 \in \mcm(\bar a_1)$, $u_1<u_2$, that for $\bar
  \Gamma_2= \DZ^{n+1}\cap \lspan\{\bar a_1\}^\bot$ we have $\dim \big(\lspan
  \bar \Gamma_2\big)= n$ and that, for any admissible pair $(\bar a_1, \bar
  a_2)$,
  \begin{equation*}
    G = \Big\{(x,x_{n+1})\mid u_1(x) < x_{n+1} < u_2(x)\Big\}
  \end{equation*} 
  is foliated by $(v_{\bar a_2,b})_{b\in \DR} \subset \mcm(\bar a_1,\bar a_2)$.
  Suppose that $u \in C^2(\DR^n)$ (possibly with self-intersections) is
  minimizing in the sense of \eqref{eq:mos_min} and $u_1<u<u_2$. Then there
  exist $w\in \mcm$ and a sequence $\bar k_i \in \bar \Gamma_2$ such that
  $T_{\bar k_i} u \to w$ in $C^1_{\loc}$.  Furthermore we have either $w \equiv
  u_1$ or $w \equiv u_2$ or $w \equiv v_{\bar a_2,b}$ for some $\bar a_2$ with
  $(\bar a_1, \bar a_2)$ admissible and $b \in \DR$.
\end{corollary}



Notice that Corollary \ref{asy} implies some kind of limit property for minimal
solutions of \eqref{eq1} (with possible self-intersections), in the sense that
it is always possible to find directions in which $u$ suitably approaches a
``pure phase'' $0$ or $1$: 
\begin{corollary}\label{asy2}
  Let $u\in C^2(\DR^n,(0,1))$ be a minimal solution of \eqref{eq1}, possibly
  with self-intersections. Then there exist a sequence $\bar k_i \in \DR^n
\times \{0\}$ and $\omega \in S^{n-1}$ in such a way that
$$ \lim_{t\rightarrow +\infty}
\lim_{i\rightarrow+\infty} u(\omega t-k_i) \in \{0,\,1\}\,.$$
\end{corollary}

\section{Proofs}

\begin{proof}[Proof of Theorem \ref{ba_dg}]
  (I) Let $b \in \DR$ be arbitrary. Without loss of generality we argue only for
  $(v_b)^+ =: w$. By Proposition \ref{4.2} we have $w \in \mcm(\bar a_1, \ldots,
  \bar a_{t-1})$.  Since the translations $T_{\bar k}$ are order preserving and
  $u_1 < v_b < u_2$, we see that $u_1 \le w \le u_2$.  Since furthermore $v_b <
  w$, we obtain $u_1< w$.
  
  We want to show that $w=u_2$.
  Suppose, by contradition, that
$w \neq u_2$.  Then Theorem \ref{order} 
implies $w<u_2$ and thus
  $u_1 < w < u_2$.  By the assumption that $G$ is foliated by $(v_b)_{b\in
    \DR}$, there exists $b_1 \in \DR$ such that $w(0) = v_{b_1}(0)$. Since
  $v_{b_1} \in \mcm(\bar a_1, \ldots, \bar a_t)$ and $w \in \mcm(\bar a_1,
  \ldots, \bar a_{t-1})$, this contradicts Theorem \ref{order}.

  (II) Let $b \in \DR$ be arbitarary. We apply Proposition \ref{4.2} (b) to $u_1
  = (v_b)^-, u_2 = (v_b)^+$ and use the fact that the translations $T_{\bar k}$
  are order preserving to infer that
  $$T_{\bar k}u > T_{\bar k}u_1 \ge u_2>  u$$
  whenever $k \in \bar \Gamma_i$ and $k\cdot \bar a_i > 0$ for $1\le i < t$.  In
  view of Proposition \ref{graphbasic} this implies $t(u) \ge t$ and $\bar
  a_i(u) = \bar a_i$ for $1\le i < t$.
      
  Suppose now that the condition $\bar a_t(u) = \bar a_t$ holds. The assumption
  that $G$ is foliated by $(v_b)_{b\in \DR}$ implies that there exists $b_0 \in
  \DR$ such that $u(0) = v_{b_0}(0)$ and Theorem \ref{order} implies $u=
  v_{b_0}$.
\end{proof}

\begin{proof}[Proof of Corollary \ref{ba_dg-1}]
Let $F:=|p|^2+W(\lfloor u\rfloor)$. It is easily seen that
$F\in C^{2,1}$, that it is periodic in $(x,u)$
and that it satisfies \eqref{Ba88et}, thence the setting
of \S\ref{Ba:set} holds
for such $F$.

Let $u_1(x):=0$ and $u_2(x):=1$.  Both $u_1$ and $u_2$ minimize the energy in
\eqref{eq2}.  They are obviously without self-intersections and so they belong
to $\mcm (\bar e_{n+1})$.

If $u_0(t)$ is
the solution of the ODE \label{ODE}
$$\ddot u_0=
u_0-3u_0^2 +2u_0^3$$ with $u_0(0)=1/2$,
$$\lim_{t\rightarrow -\infty}u_0(t)=0 \;{\mbox{ and }}\;
\lim_{t\rightarrow +\infty}u_0(t)=1\,,$$ then for every $\omega \in S^{n-1}$,
the family $(v_{\omega,b})_{b\in \DR}$ with
\begin{equation}
  \label{vomb}
  v_{\omega,b}(x)=u_0(\omega\cdot x-b)  
\end{equation}
for any $x\in \DR^n$ is an extremal, and consequently minimal foliation of
the set
$$\DR^n\times
(0,1)=\{(x,x_{n+1}) \mid x \in \DR^n, u_1(x)<x_{n+1}<u_2(x)\}\,.$$ A reference
for the fact that extremal foliations are actually minimal foliations is
e.g.~\cite[6.3]{GH1}.  One readily checks that $v_{\omega,b}\in\mcm (\bar
e_{n+1},\bar\omega)$, where $\bar \omega = (\omega,0) \in \DR^n \times \{0\}$.

Let $u$ be as requested by Corollary \ref{ba_dg-1}.
Since $u$ is bounded, $\bar a_1(u)=\bar e_{n+1}$ and so
\begin{equation}\label{AAU}
  \bar \Gamma_2(u) = \DZ^{n+1}\cap
  (\lspan \bar a_1(u))^\bot\subseteq \DR^n\times\{0\}\,
  .\end{equation}
It then follows from
Theorem \ref{ba_dg}\eqref{II} that $t(u) \ge 2$, and in particular $u$ is
non-constant. Furthermore, for $\bar \omega = \bar a_2(u)$ there exists $b_0 \in
\DR$ such that $u = v_{\omega,b_0}$. Thus the level sets of $u$ are hyperplanes.

\end{proof}

\begin{proof}[Proof of Corollary \ref{ba_dg-2}]
Let $u \in C^2(\DR^n,(0,1))$ be a minimal solution of
\eqref{eq1}.
If
Problem {\bf \cite{Ba89}} has a positive answer in dimension $n$
for $F=|p|^2+W(\lfloor u\rfloor)$ 
and $\bar a_1=\bar e_{n+1}$, then $u$
is without self-intersections.

Consequently,
by Corollary \ref{ba_dg-1},
all the level sets of $u$ are hyperplanes, giving
that
Problem {\bf \cite{DG}$_{MIN}$} has a positive 
answer in dimension $n$.
\end{proof}

\begin{proof}[Proof of Corollary \ref{asy}]
  By \cite[Theorem 3.1, Corollary 3.2]{Mos}, we have that $\sup |u_x| < \infty$
  and so, by \cite[Theorem (8.1)]{Ba89}, there exists a sequence $\bar k_i \in
  \bar\Gamma_2$ such that $T_{\bar k_i} u \to w \in \mcm$ in $C^1_{\loc}$.
  Clearly, $u_1\le w\le u_2$.  We may suppose indeed that $u_1<w<u_2$, otherwise
  we get one of the first two alternatives in the statement of Corollary
  \ref{asy}. By Theorem \ref{ba_dg} (II) we have $\bar a_1(w) = \bar a_1$ and
  $t(w) \ge 2$. Furthermore together with the assumptions of Corollary \ref{asy}
  Theorem \ref{ba_dg} (II) implies that $t(w) = 2$ and that for some $b \in \DR$
  we have $w = v_{\bar a_2, b}$, where $\bar a_2 = \bar a_2(w)$.
\end{proof}

\begin{proof}[Proof of Corollary \ref{asy2}]
  With $u_1(x):=0$, $u_2(x):=1$, $v_{\omega,b}$ as in \eqref{vomb} and $G =
  \DR^n \times (0,1)$ the assumptions of Corollary \ref{asy} are satisfied.
\end{proof}


\addcontentsline{toc}{chapter}{Bibliography}
\bibliographystyle{amsalpha}    
\bibliography{lit}            

\end{document}